\documentclass[reqno, 11pt]{amsart}
\usepackage[dvipsnames,usenames]{color}
\usepackage[colorlinks=true, urlcolor=NavyBlue, linkcolor=NavyBlue, citecolor=NavyBlue]{hyperref}
\usepackage{graphicx}
\usepackage{epsfig}
\usepackage[latin1]{inputenc}
\usepackage{amsmath}
\usepackage{amsfonts}
\usepackage{amssymb}
\usepackage{amsthm}
\usepackage{amscd}
\usepackage{verbatim}
\usepackage{subfigure}
\usepackage{pinlabel}
\usepackage{stmaryrd}
\usepackage{enumerate, enumitem}
\usepackage{multirow}

\usepackage{tikz}
\usetikzlibrary{arrows}
\usetikzlibrary{decorations.pathreplacing}
\usepackage{verbatim}
	
\newtheorem{theorem}{Theorem}[section]
\newtheorem{Theorem}{Theorem}
\newtheorem{lemma}[theorem]{Lemma}
\newtheorem{proposition}[theorem]{Proposition}

\newtheorem{corollary}[theorem]{Corollary}

\theoremstyle{definition}
\newtheorem{definition}[theorem]{Definition}

\theoremstyle{remark}
\newtheorem{remark}[theorem]{Remark}
\newtheorem{example}[theorem]{Example}

\def\F{\mathbb{F}}
\def\N{\mathbb{N}}
\def\Q{\mathbb{Q}}
\def\R{\mathbb{R}}
\def\bbT{\mathbb{T}}
\def\Z{\mathbb{Z}}

\def\cC{\mathcal{C}}

\def\cH{\mathcal{H}}

\def\cM{\mathcal{M}}

\def\gr{\mathrm{gr}}

\def\bfx{\mathbf{x}}
\def\bfy{\mathbf{y}}

\def\mfs{\mathfrak{s}}
\def\mft{\mathfrak{t}}

\def\CFKi{CFK^{\infty}}
\def\CFKh{\widehat{CFK}}
\def\cCFK{\mathcal{CFK}}

\def\bfalpha{\boldsymbol{\alpha}}
\def\bfbeta{\boldsymbol{\beta}}

\def\horz{\textup{horz}}
\def\vert{\textup{vert}}
\def\min{\textup{min}}
\def\max{\textup{max}}

\def\dim{\textup{dim}}

\def\sgn{\textup{sgn}}

\def\rank{\operatorname{rank}}
\def\coker{\operatorname{coker}}
\def\sgn{\operatorname{sgn}}

\def\d{\partial}
\def\varep{\varepsilon}
\def\Vl{\underline{V}\,_0}
\def\Vu{\overline{V}_0}
\def\dl{\underline{d}}
\def\du{\overline{d}}

\newcommand{\co}{\mskip0.5mu\colon\thinspace}

\title{A survey on Heegaard Floer homology and concordance}

\subjclass[2013]{}
\author[Jennifer Hom]{Jennifer Hom}
\thanks{The author was partially supported by NSF grants DMS-1128155, DMS-1307879, and a Sloan Research Fellowship.}
\address {School of Mathematics, Georgia Institute of Technology, Atlanta, GA 30332}
\address{School of Mathematics, Institute for Advanced Study, Princeton, NJ 08540}
\email{hom@math.gatech.edu}

\numberwithin{equation}{section}

\begin{document}

\begin{abstract}
In this survey article, we discuss several different knot concordance invariants coming from the Heegaard Floer homology package of Ozsv\'ath and Szab\'o. Along the way, we prove that if two knots are concordant, then their knot Floer complexes satisfy a certain type of stable equivalence.
\end{abstract}

\maketitle

\section{Introduction}
Two knots, $K_1$ and $K_2$, are \emph{concordant} if they cobound a smooth, properly embedded cylinder in $S^3 \times [0, 1]$. The \emph{knot concordance group}, denoted $\cC$, consists of knots in $S^3$ modulo concordance, with the operation induced by connected sum. See \cite{Livingston} for a survey of knot concordance. In this article, we discuss several concordance invariants coming from the Heegaard Floer homology package of Ozsv\'ath and Szab\'o \cite{OS3manifolds1}.

Heegaard Floer homology has proved to be an effective tool in understanding three-manifolds and knots inside of them. For example, Heegaard Floer homology detects fiberedness of knots and three-manifolds \cite{Ghiggini, Ni, Nifibred3}, and bounds the four-ball genus and unknotting number \cite{OS4ball, OSunknotting}. 
In its simplest form, the theory associates to every closed three-manifold $Y$ a graded chain complex $\widehat{CF}(Y)$, whose chain homotopy type is an invariant of the manifold; more refined invariants, such as $CF^-(Y)$, associate a complex of $\F[U]$-modules to each three-manifold. The rank and gradings of the homology of these chain complexes contain geometric information about the manifold. Furthermore, a four-manifold cobordism between two closed three-manifolds induces a homomorphism between the Heegaard Floer homology groups of the three-manifolds.

There are many generalizations and refinements of Heegaard Floer homology. For example, a knot $K$ in a three-manifold $Y$ induces a filtration on $\widehat{CF}(Y)$, and the homology of the associated graded object is known as knot Floer homology \cite{OSknots, R}. This invariant categorifies the Alexander polynomial in the sense that the Alexander polynomial is the graded Euler characteristic of knot Floer homology. The knot $K$ similarly induces a filtration on the more refined invariant $CF^-(Y)$. As we will see, many concordance invariants can be extracted from this filtered chain complex and the closely related $\Z \oplus \Z$-filtered complex $\CFKi(K)$.

The filtered chain homotopy type of $\CFKi(K)$ is an invariant of the isotopy type of $K$. We will be interested in the properties of $\CFKi(K)$ which are invariants of the concordance class of $K$. For example, the integer-valued concordance invariant $\tau$ of Ozsv\'ath and Szab\'o \cite{OS4ball}, the $\{-1, 0, 1\}$-valued concordance invariant $\varep$ of the author \cite{Homsmooth}, and the $\Upsilon(t)$ invariant of Ozsv\'ath-Stipsicz-Szab\'o \cite{OSS} are all determined by $\CFKi(K)$. Indeed, although not the original proof, the fact that these invariants are all concordance invariants can be seen as a corollary of the following theorem.

\begin{Theorem}\label{thm:main}
If $K_1$ and $K_2$ are concordant, then we have the following filtered chain homotopy equivalence:
\[ \CFKi(K_1) \oplus A_1 \simeq \CFKi(K_2) \oplus A_2, \]
where $A_1, A_2$ are acyclic, i.e., $H_*(A_1)=H_*(A_2)=0$. In particular, if $K$ is slice, then 
\[ \CFKi(K) \simeq \CFKi(U) \oplus A, \]
where $U$ denotes the unknot and $A$ is acyclic.
\end{Theorem}

\begin{remark}
Note that in general, the $A_i$ are not acyclic as \emph{filtered} chain complexes. See, for example, Figure \ref{fig:figure8}.
\end{remark}

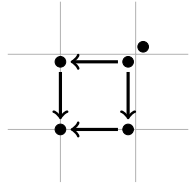
\begin{figure}[htb!]
\centering
\begin{tikzpicture}[scale=1.0]
	\draw[step=1, black!30!white, very thin] (-0.7, -0.7) grid (1.7, 1.7);
	
	\filldraw (0.9, 0.9) circle (2pt) node[] (a){};
	\filldraw (0.9, 0) circle (2pt) node[] (b){};
	\filldraw (0, 0.9) circle (2pt) node[] (c){};
	\filldraw (0, 0) circle (2pt) node[] (e){};
	\filldraw (1.1, 1.1) circle (2pt) node[] (x){};
	\draw [very thick, <-] (b) -- (a);
	\draw [very thick, <-] (c) -- (a);
	\draw [very thick, <-] (e) -- (b);
	\draw [very thick, <-] (e) -- (c);
\end{tikzpicture}
\caption{The knot Floer complex of the figure eight knot, which is generated over $\F[U, U^{-1}]$ by the five generators above, four of which generate an acyclic summand.}
\label{fig:figure8}
\end{figure}

\subsection*{Outline} In Section \ref{sec:HF}, we give a brief overview of Heegaard Floer homology for closed three-manifolds and the knot Floer complex for knots in $S^3$. In Section \ref{sec:concinvts}, we review the following concordance invariants coming from Heegaard Floer theory:
\begin{itemize}
	\item Ozsv\'ath-Szab\'o's $\Z$-valued invariant $\tau$ \cite{OS4ball},
	\item Ozsv\'ath-Szab\'o's $\Z$-valued invariant $\nu$ \cite{OSrational},
	\item the author's $\{-1, 0, 1\}$-valued invariant $\varep$ \cite{Homsmooth},
	\item the author and Wu's $\nu^+$ (equivalently, $\nu^-$) invariant \cite{HomWu}, based on work of Rasmussen \cite{R},
	\item the sequence of invariants $V_s$, coming from the surgery formula of Ozsv\'ath-Szab\'o \cite{OSinteger}, and the closely related $d$-invariants of surgery \cite{OSabsgr},
	\item Hendricks-Manolescu's $\Vl$ and $\Vu$ invariants, coming from involutive Heegaard Floer homology \cite{HendricksManolescu},
	\item $d$-invariants of branched covers (due to Manolescu-Owens \cite{ManolescuOwens} for the case of double-branched covers, and Jabuka \cite{Jabuka} for prime-powered branched covers),
	\item Ozsv\'ath-Stipsicz-Szab\'o's $\Upsilon(t)$ \cite{OSS}, which gives a homomorphism from $\cC$ to the group of piecewise-linear functions on the interval $[0,2]$.
\end{itemize}
The proof of Theorem \ref{thm:main} is in Subsection \ref{subsec:HFsurgeryinvts}. We discuss applications to concordance genus in Section \ref{sec:concgen}. We conclude with a comparison of some of these invariants in Section \ref{sec:comparison}. Throughout, we work with $\F=\Z/2\Z$ coefficients. 

\subsection*{Acknowledgements} I would like to thank Chuck Livingston and Kouki Sato for helpful discussions.

\section{The Heegaard Floer homology package}\label{sec:HF}


\subsection{Heegaard Floer homology}
We give a brief overview of the definition of Heegaard Floer homology. For more details, see the expository articles \cite{OSsurvey, OSsurvey2}, and for complete details, see \cite{OSknots}.

A \emph{pointed Heegaard diagram} is a tuple $\cH = (\Sigma, \bfalpha, \bfbeta, w)$ where
\begin{itemize}
	\item $\Sigma$ is a closed, oriented surface of genus $g$,
	\item $\bfalpha$ (respectively $\bfbeta$) is a $g$-tuple of pairwise disjoint circles $\alpha_1, \dots, \alpha_g$ (respectively $\beta_1, \dots, \beta_g$) in $\Sigma$,
	\item $w$ is a point in the complement of $\bfalpha$ and $\bfbeta$,
	\item $\bfalpha$ and $\bfbeta$ each span a $g$-dimensional subspace of $H_1(\Sigma; \Z)$.
\end{itemize}
We further require that the intersections between the $\alpha$- and $\beta$-circles be transverse. A pointed Heegaard diagram specifies a three-manifold: thicken $\Sigma$, attach thickened disks along each of the $\alpha$- and $\beta$-circles, and cap off each of the two remaining boundary components with three-balls.

To a pointed Heegaard diagram $\cH$ describing a three-manifold $Y$, we associate a chain complex of $\F[U]$-modules where $U$ is a formal variable. Since the chain homotopy type of this complex is independent of the choice of Heegaard diagram for $Y$, we will abuse notation slightly and denote this chain complex by $CF^-(Y)$.

Let $\textup{Sym}^g(\Sigma)$ denote the $g$-fold symmetric product of $\Sigma$, that is, 
\[ \textup{Sym}^g(\Sigma) = \Sigma \times \dots \times \Sigma / S_g, \]
where $S_g$ denotes the symmetric group on $g$ letters. Consider the subspaces
\[ \bbT_\alpha = \alpha_1 \times \dots \times \alpha_g \qquad \textup{ and } \qquad  \bbT_\beta = \beta_1 \times \dots \times \beta_g. \]
The chain complex is generated over $\F[U]$ by points of $\bbT_\alpha \cap \bbT_\beta$, or, equivalently, by unordered $g$-tuples of points in $\bfalpha \cap \bfbeta$ such that each $\alpha$- and each $\beta$-circle is used exactly once. As is standard, given $\bfx, \bfy \in\bbT_\alpha \cap \bbT_\beta$, we let $\pi_2(\bfx, \bfy)$ denote the space of homotopy classes of disks connecting $\bfx$ and $\bfy$. Associated to $\phi \in \pi_2(\bfx, \bfy)$ are two integers: the \emph{Maslov index} $\mu(\phi)$ and the \emph{multiplicity at} $w$ $n_w(\phi)$, which is the algebraic intersection number of $\phi$ with $\{w\} \times \textup{Sym}^{g-1}(\Sigma)$. See \cite[Section 2.1]{OSsurvey} for details.

The differential is
\[ \d^- \bfx = \sum_{\bfy \in \bbT_\alpha \cap \bbT_\beta} \sum_{\{ \phi \in \pi_2(\bfx, \bfy) \mid \mu(\phi)=1\} } \# \Big(\frac{\cM(\phi)}{\R}\Big) U^{n_w(\phi)} \bfy, \]
where $\cM(\phi)$ denotes the moduli space of pseudo-holomorphic representatives of $\phi$. (Note that when $b_1(Y)>0$, there is an admissibility condition for Heegaard diagrams; for the purposes of this survey, we will primarily be concerned with rational homology spheres.) When $\mu(\phi)=1$, the space $\cM(\phi)$ is one-dimensional, and $\# (\frac{\cM(\phi)}{\R})$ is the mod $2$ count of points in the factor space $\frac{\cM(\phi)}{\R}$. The complex $CF^-(Y)$ splits along spin$^c$ structures (which are in bijection with $H^2(Y; \Z)$) on $Y$ and we write the homology of $CF^-(Y)$ as
\[ HF^-(Y) = \bigoplus_{\mfs \in \textup{spin}^c(Y)} HF^-(Y, \mfs). \]

Let $Y$ be a rational homology sphere. We can define a relative $\Z$-grading on $CF^-(Y, \mfs)$ as follows:
\[ M(\bfx) - M(\bfy) = \mu(\phi) -2n_w(\phi) \]
where $\phi \in \pi_2(\bfx, \bfy)$. Multiplication by $U$ lowers the grading by 2, i.e.,
\[ M(U \bfx) = M(\bfx) - 2.\]
Moreover, this relative $\Z$-grading can be lifted to an absolute $\Q$-grading \cite[Theorem 7.1]{OSsmoothfour}. The \emph{correction term} or \emph{$d$-invariant}, $d(Y, \mfs)$, is the maximal $\Q$-grading of a non-$U$-torsion element in $HF^-(Y, \mfs)$ and is an invariant of spin$^c$ rational homology cobordism. (See \cite[Definition 1.1]{OSabsgr} for the definition of the spin$^c$ rational homology cobordism group.)

There is also a simpler flavor of Heegaard Floer homology, $\widehat{HF}$, which takes the form of a graded $\F$-vector space, rather than a graded $\F[U]$-module. The generators of the chain complex $\widehat{CF}$ are again points in $\bbT_\alpha \cap \bbT_\beta$, and the differential is now defined to be
\[ \widehat\d \bfx = \sum_{\bfy \in \bbT_\alpha \cap \bbT_\beta} \sum_{\{ \phi \in \pi_2(\bfx, \bfy) \mid \mu(\phi)=1, n_w(\phi)=0\} } \# \Big(\frac{\cM(\phi)}{\R}\Big) \bfy. \]
Equivalently, the complex $(\widehat{CF}, \widehat\d$) is obtained from $(CF^-, \d^-)$ by setting $U=0$. For a rational homology sphere $Y$, we have the inequality
\[ \dim \; \widehat{HF}(Y) \geq |H_1(Y; \Z)|. \]
If equality is achieved, then $Y$ is called an \emph{L-space}.

One can also consider the complex $(CF^\infty, \d^\infty)$, which is obtained from $(CF^-, \d^-)$ by tensoring over $\F[U]$ with $\F[U, U^{-1}]$, and the complex $(CF^+, \d^+)$, which is the quotient of $CF^\infty$ by $CF^-$. For a rational homology sphere $Y$, we have the isomorphism $HF^\infty(Y, \mfs) \cong \F[U, U^{-1}]$ (see \cite[Theorem 10.1]{OS3manifolds2} for the statement for $Y$ with $b_1(Y)\neq 0$), so $HF^\infty$ does not contain any new information, while $HF^\pm$ and $\widehat{HF}$ in general do.

\begin{example}
From the standard genus one Heegaard diagram for $S^3$, we see that
\[ HF^-(S^3) \cong \F[U], \qquad \widehat{HF}(S^3) \cong \F, \qquad \textup{ and } \qquad HF^\infty(S^3) \cong \F[U, U^{-1}]. \]
The absolute grading is defined such that $\F$ in both $HF^-(S^3)$ and $\widehat{HF}(S^3)$ has Maslov grading zero.
\end{example}

The Heegaard Floer homology package has the structure of a $(3+1)$-TQFT. That is, a four-manifold spin$^c$ cobordism $(W, \mft)$ between $Y_1$ and $Y_2$ induces a map 
\[ F^\circ_{W, \mft} \co HF^\circ(Y_1, \mfs_1) \rightarrow HF^\circ(Y_2, \mfs_2), \]
where $\circ$ can denote $-, +, \widehat{\;\;},$ or $\infty$ and $\mfs_i = \mft|_{Y_i}$. These maps will play a key role in Section \ref{subsec:HFsurgeryinvts}.

\subsection{The knot Floer complex}
A knot $K$ in a three-manifold $Y$ induces a filtration on $CF^\circ$, as in \cite{OSknots}. Here, we will restrict ourselves to the case of $Y=S^3$. For an expository overview of knot Floer homology, we refer the reader to \cite{Manolescuknots} and for complete details, see \cite{OSknots}.

A \emph{doubly pointed Heegaard diagram} for a knot $K \subset S^3$ is a tuple $\cH = (\Sigma, \bfalpha, \bfbeta,w, z)$ where $(\Sigma, \bfalpha, \bfbeta,w)$ is a pointed Heegaard diagram for $S^3$ and $z$ is a point in the complement of $\bfalpha$ and $\bfbeta$. The knot $K$ is specified by connecting $z$ to $w$ in the complement of the $\bfalpha$ disks and $w$ to $z$ in the complement of the $\bfbeta$ disks. See Figure \ref{fig:HDtrefoil}.

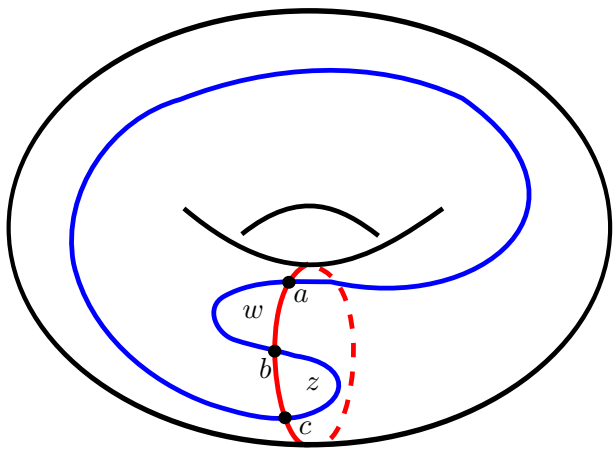
\begin{figure}[ht]
\centering
\labellist
\pinlabel {$w$} at 62 34
\pinlabel {$z$} at 77 16
\pinlabel {$a$} at 74 38
\pinlabel {$b$} at 65 20
\pinlabel {$c$} at 75 5
\endlabellist
\includegraphics[scale=1.5]{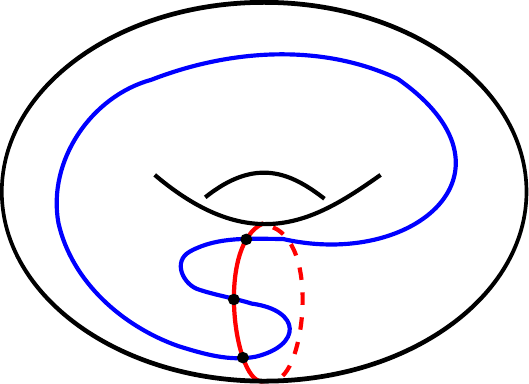}
\caption{A Heegaard diagram for the right-handed trefoil, $T_{2,3}$.}
\label{fig:HDtrefoil}
\end{figure}

The chain complex $CFK^-(K)$ is a freely generated chain complex over $\F[U]$ by points in $\bbT_\alpha \cap \bbT_\beta$, and the differential $\d^-$ is, as before,
\[ \d^- \bfx = \sum_{\bfy \in \bbT_\alpha \cap \bbT_\beta} \sum_{\{ \phi \in \pi_2(\bfx, \bfy) \mid \mu(\phi)=1\} } \# \Big(\frac{\cM(\phi)}{\R}\Big) U^{n_w(\phi)} \bfy. \]
The fact that $HF^-(S^3)=\F[U]_{(0)}$ (where the subscript $(0)$ denotes the grading of $1 \in \F[U]$) determines the absolute Maslov grading on $CFK^-(K)$.

The additional basepoint $z$ is used to define a $\Z$-filtration, called the \emph{Alexander filtration}, on $CFK^-$. Given generators $\bfx$ and $\bfy$ and $\phi \in \pi_2(\bfx, \bfy)$, the relative Alexander grading is defined to be
\[ A(\bfx)-A(\bfy)=n_z(\phi)-n_w(\phi). \]
The differential respects the Alexander filtration and multiplication by $U$ lowers the Alexander filtration by one:
\[ A(\d^- \bfx) \leq A(\bfx) \qquad \textup{ and } \qquad A(U \bfx) = A(\bfx)-1. \]
The absolute Alexander grading $s$ is determined by the fact that $\max \{s \mid \widehat{HFK}(K, s) \neq 0 \} = - \min  \{s \mid \widehat{HFK}(K, s) \neq 0 \}$, where $\widehat{HFK}(K,s)$ is as defined below.

By setting $U=0$ in $CFK^-$, we obtain the $\Z$-filtered chain complex $\widehat{CFK}$ and by tensoring over $\F[U]$ with $\F[U, U^{-1}]$, we obtain the $\Z$-filtered chain complex $\CFKi$, called the \emph{knot Floer complex}.
In addition to the Alexander filtration, there is a second filtration, called the \emph{algebraic filtration}, on $\CFKi$ given by the (negative of the) $U$-exponent; namely, $U^n \bfx$ has algebraic filtration level $-n$. Thus, the knot Floer complex $\CFKi$ has the structure of a $\Z \oplus \Z$-filtered chain complex. See \cite[Section 2.6]{OSrational} for more details on $\Z \oplus \Z$-filtered chain complexes. We denote the two filtrations in the plane, with the Alexander filtration along the vertical axis and the algebraic filtration along the horizontal axis. We suppress the Maslov grading from the picture.

\begin{example}
The knot Floer complex of the right-handed trefoil, $\CFKi(T_{2,3})$, is freely generated over $\F[U, U^{-1}]$ by $a, b,$ and $c$, as in Figure \ref{fig:HDtrefoil}. We have that $\d b= Ua + c$ and $\d a = \d c =0$. The Maslov gradings of the generators are $M(a)=0$, $M(b)=-1$, and $M(c)=-2$, and the Alexander gradings are $A(a)=1, A(b)=0$, and $A(c)=-1$. See Figure \ref{fig:CFKinftyT23}.
\end{example}

\begin{figure}[htb!]
\centering
\vspace{5pt}
\subfigure[]{
\begin{tikzpicture}[scale=1.2]
	\begin{scope}[thin, gray]
		\draw [<->] (-3, 0) -- (3, 0);
		\draw [<->] (0, -3) -- (0, 3);
	\end{scope}
	\draw[step=1, black!30!white, very thin] (-2.9, -2.9) grid (2.9, 2.9);
	\filldraw (1, 2) circle (2pt) node[] (a1){};
	\filldraw (2, 2) circle (2pt) node[] (b1){};
	\filldraw (2, 1) circle (2pt) node[] (c1){};
	\draw [very thick, <-] (a1) -- (b1);
	\draw [very thick, <-] (c1) -- (b1);
	\node [left] at (a1) {$U^{-1}a$};
	\node [above] at (b1) {$U^{-2}b$};
	\node [below] at (c1) {$U^{-2}c$};
	
	\filldraw (0, 1) circle (2pt) node[] (a2){};
	\filldraw (1, 1) circle (2pt) node[] (b2){};
	\filldraw (1, 0) circle (2pt) node[] (c2){};
	\draw [very thick, <-] (a2) -- (b2);
	\draw [very thick, <-] (c2) -- (b2);
	\node [left] at (a2) {$a$};
	\node [above] at (b2) {$U^{-1}b$};
	\node [below] at (c2) {$U^{-1}c$};	
	
	\filldraw (-1, 0) circle (2pt) node[] (a3){};
	\filldraw (0, 0) circle (2pt) node[] (b3){};
	\filldraw (0, -1) circle (2pt) node[] (c3){};
	\draw [very thick, <-] (a3) -- (b3);
	\draw [very thick, <-] (c3) -- (b3);
	\node [left] at (a3) {$Ua$};
	\node [above] at (b3) {$b$};
	\node [below] at (c3) {$c$};

	\filldraw (-2, -1) circle (2pt) node[] (a4){};
	\filldraw (-1, -1) circle (2pt) node[] (b4){};
	\filldraw (-1, -2) circle (2pt) node[] (c4){};
	\draw [very thick, <-] (a4) -- (b4);
	\draw [very thick, <-] (c4) -- (b4);
	\node [left] at (a4) {$U^2a$};
	\node [above] at (b4) {$Ub$};
	\node [below] at (c4) {$Uc$};
	
	\filldraw (2.5, 2.5) circle (0.8pt);
	\filldraw (2.6, 2.6) circle (0.8pt);
	\filldraw (2.7, 2.7) circle (0.8pt);
	
	\filldraw (-2.2, -2.2) circle (0.8pt);
	\filldraw (-2.3, -2.3) circle (0.8pt);
	\filldraw (-2.4, -2.4) circle (0.8pt);
\end{tikzpicture}
\label{fig:CFKinftyT23}
}
\\
\subfigure[]{
\begin{tikzpicture}[scale=1.2]
	\begin{scope}[thin, gray]
		\draw [<->] (-3, 0) -- (1, 0);
		\draw [<->] (0, -3) -- (0, 3);
	\end{scope}
	\draw[step=1, black!30!white, very thin] (-2.9, -2.9) grid (0.9, 2.9);
	
	\filldraw (0, 1) circle (2pt) node[] (a2){};
	\node [left] at (a2) {$a$};
	
	\filldraw (-1, 0) circle (2pt) node[] (a3){};
	\filldraw (0, 0) circle (2pt) node[] (b3){};
	\filldraw (0, -1) circle (2pt) node[] (c3){};
	\draw [very thick, <-] (a3) -- (b3);
	\node [left] at (a3) {$Ua$};
	\node [above] at (b3) {$b$};
	\node [below] at (c3) {$c$};

	\filldraw (-2, -1) circle (2pt) node[] (a4){};
	\filldraw (-1, -1) circle (2pt) node[] (b4){};
	\filldraw (-1, -2) circle (2pt) node[] (c4){};
	\draw [very thick, <-] (a4) -- (b4);
	\node [left] at (a4) {$U^2a$};
	\node [above] at (b4) {$Ub$};
	\node [below] at (c4) {$Uc$};

	\filldraw (-3, -2) circle (2pt) node[] (a5){};
	\filldraw (-2, -2) circle (2pt) node[] (b5){};
	\draw [very thick, <-] (a5) -- (b5);
	\node [left] at (a5) {$U^3a$};
	\node [above] at (b5) {$U^2b$};
	
	\filldraw (-2.6, -2.8) circle (0.8pt);
	\filldraw (-2.7, -2.9) circle (0.8pt);
	\filldraw (-2.8, -3) circle (0.8pt);
\end{tikzpicture}
\label{fig:CFKminusT23}
}
\hspace{100pt}
\subfigure[]{
\begin{tikzpicture}[scale=1.2]
	\begin{scope}[thin, gray]
		\draw [<->] (0, -3) -- (0, 3);
	\end{scope}
	\draw[step=1, black!30!white, very thin] (-0.1, -2.9) grid (0.1, 2.9);
	\filldraw (0, 1) circle (2pt) node[] (a){};
	\filldraw (0, 0) circle (2pt) node[] (b){};
	\filldraw (0, -1) circle (2pt) node[] (c){};
	\draw [very thick, <-] (c) -- (b);
	\node [left] at (a) {$a$};
	\node [left] at (b) {$b$};
	\node [left] at (c) {$c$};
\end{tikzpicture}
\label{fig:CFKhatT23}
}
\caption{Top, $\CFKi(T_{2,3})$. The arrows indicate the differential, e.g., $\d b = Ua + c$. Bottom left, $\oplus_{s \in \Z} C\{ i \leq 0, j= s\}$. Bottom right, the $\Z$-filtered chain complex $\CFKh(T_{2,3})=C\{i=0\}$. The Maslov gradings are $M(a)=0$, $M(b)=-1$, and $M(c)=-2$.}
\label{fig:}
\end{figure}
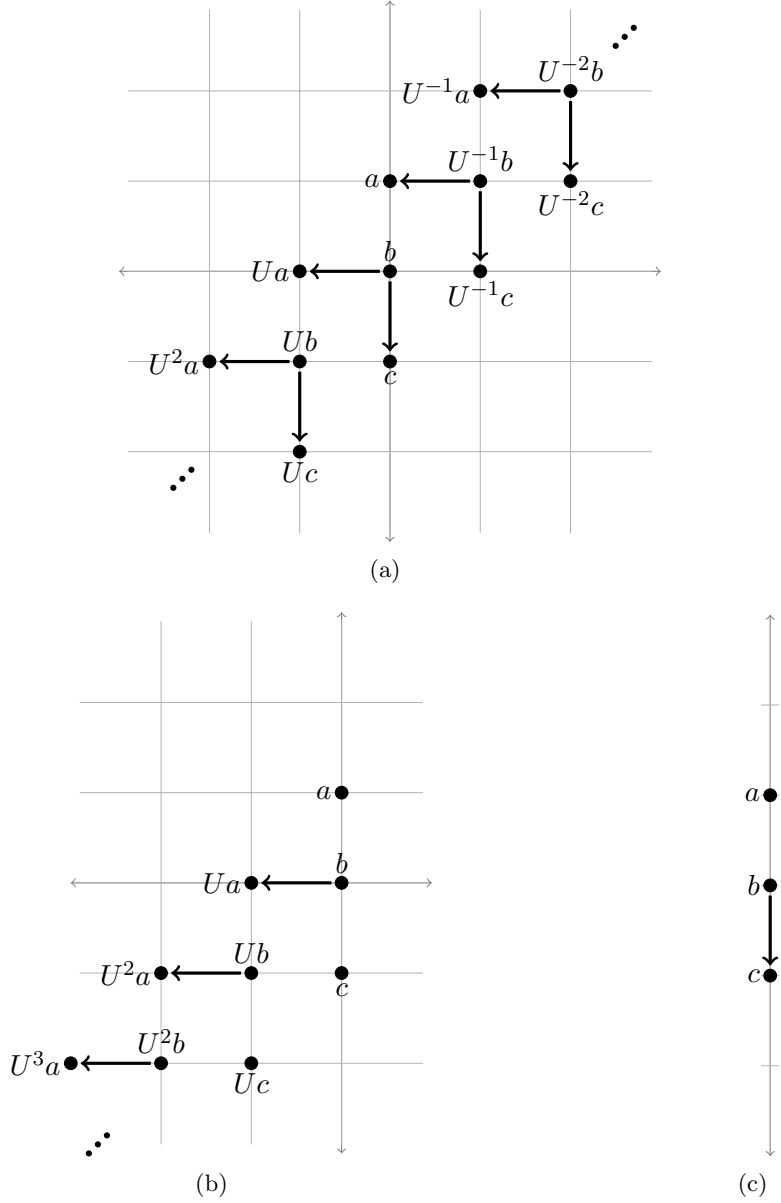

We will often be interested in various subcomplexes, quotient complexes, and subquotient complexes of $\CFKi(K)$. Given a subset $X \subset \Z \oplus \Z$, let $CX$ denote the complex generated by elements with filtration level $(i, j) \in X$. The complex $CX$ will naturally inherit the structure of a chain complex if it is a subquotient (i.e., a quotient of a subcomplex) of $\CFKi(K)$. Note that $CX$ is a filtered subcomplex exactly when $X$ has the property that $(i, j) \in X$ implies that $(i', j') \in X$ for all pairs $(i', j')$ such that $i' \leq i$ and $j' \leq j$.

Note that the subcomplex $C\{i \leq 0\}$, which is naturally $\Z$-filtered by the Alexander filtration, is equal to $CFK^-(K)$. In particular, $H_*(C\{i \leq 0\}) = HF^-(S^3) \cong \F[U]$. One may consider the associated graded complex, 
\begin{equation}\label{eqn:assocgr}
\bigoplus_{s \in \Z} C\{ i \leq 0, j=s\} = \bigoplus_{s \in \Z} C\{ i \leq 0, j\leq s\} / C\{ i \leq 0, j < s\},
\end{equation}
whose homology we denote by $HFK^-(K)$:
\begin{align*}
	HFK^-(K) &= \bigoplus_{s \in \Z} HFK^-(K,s) \\
	&= \bigoplus_{s \in \Z} H_*\Big(C\{ i \leq 0, j=s\}\Big).
\end{align*}
In passing to the associated graded complex, the Alexander filtration becomes the Alexander grading. There is a natural $U$-action 
\[ U \co HFK^-(K, s) \rightarrow HFK^-(K, s-1), \]
induced by the chain map
	\[ U \co C\{i \leq 0, j=s\} \rightarrow C\{i\leq 0, j=s-1\}. \]
Since $K$ induces a filtration on $CF^-(S^3)$, there is a spectral sequence from the associated graded complex in \eqref{eqn:assocgr} to $HF^-(S^3)$. The $E_1$ page of the spectral sequence is the homology of the associated graded complex, i.e., $HFK^-$.

At times, we will be interested in the \emph{horizontal differential} $\d^\horz$, that is, the part of the differential that preserves the Alexander filtration. Similarly, we will also be interested in the \emph{vertical differential} $\d^\vert$, that is, the part of the differential that preserves the algebraic filtration.

\begin{example}
From Figure \ref{fig:CFKminusT23}, we see that $HFK^-(T_{2,3},1)$ is generated by $a$ and that $HFK^-(T_{2,3},s)$ is generated by $U^{-s-1}c$ for $s \leq -1$. Thus, we have
\[	HFK^-(T_{2,3},s) = 
		\begin{cases}
			\F_{(0)} \quad & \text{if } s=1 \\
			\F_{(2s)} \quad & \text{if } s \leq -1 \\
			0 \quad & \text{otherwise}.
		\end{cases}
\]
The $U$-action maps $HFK^-(T_{2,3},s)$ isomorphically onto $HFK^-(T_{2,3},s-1)$ for $s \leq -1$ and is zero otherwise.
\end{example}

The subquotient complex $C\{i = 0\} = C\{i \leq 0\} / C\{i<0\}$, which is naturally $\Z$-filtered by the Alexander filtration, is equal to $\widehat{CFK}(K)$. We have that $H_*(C\{i = 0\}) = \widehat{HF}(S^3) \cong \F$. The homology of the associated graded object of $\widehat{CFK}(K)$ is
\begin{align*}
	 \widehat{HFK}(K) &= \bigoplus_{s \in \Z} \widehat{HFK}(K, s) \\
	 &= \bigoplus_{s \in \Z} H_*\Big(C\{i=0, j=s\}\Big)
\end{align*}
and is referred to as \emph{knot Floer homology}. There is a spectral sequence from the associated graded complex $\oplus_{s \in \Z}C\{i=0, j=s\}$ to $\widehat{HF}(S^3)$; the $E_1$ page is $\widehat{HFK}$. The fact that $H_*(\widehat{CFK}(K)) = \widehat{HF}(S^3) \cong \F$ will play a key role in the definition of the concordance invariant $\tau$ in Section \ref{sec:tau}. Knot Floer homology is bigraded, by the Maslov grading $M$ and the Alexander grading $s$, and the graded Euler characteristic equals the Alexander polynomial of $K$:
	\[ \Delta_K(t) = \sum_{M,s} (-1)^M t^s \rank \widehat{HFK}_M(K, s). \]
	
\begin{example}
From Figure \ref{fig:CFKhatT23}, we see that the differential on the associated graded complex is zero, since the only differential on $\widehat{CFK}(T_{2,3})$ decreases filtration. Thus,
\[	\widehat{HFK}(T_{2,3},s) = 
		\begin{cases}
			\F_{(0)} \quad & \text{if } s=1 \\
			\F_{(-1)} \quad & \text{if } s=0 \\
			\F_{(-2)} \quad & \text{if } s=-1 \\
			0 \quad & \text{otherwise}.
		\end{cases}
\]
\end{example}

Let $-K$ denote the reverse of the mirror image of $K$. We have that 
\[ \CFKi(-K) \simeq \CFKi(K)^* \]
where $\CFKi(K)^*$ denotes the dual of $\CFKi(K)$, i.e., $\textup{Hom}_{\F[U, U^{-1}]}(\CFKi(K), \F[U, U^{-1}])$, and
\[ \CFKi(K_1 \# K_2) \simeq \CFKi(K_1) \otimes_{\F[U, U^{-1}]} \CFKi(K_2), \]
by \cite[Section 3.5]{OSknots} and \cite[Theorem 7.1]{OSknots} respectively. These properties will be of importance at the end of Section \ref{sec:hatmaps}. In terms of illustrating $\CFKi(K)$ in the plane, as in Figure \ref{fig:CFKinftyT23}, we can obtain $\CFKi(K)^*$ by rotating the page $180^\circ$ and reversing the directions of all of the arrows. We also have that the $\Z \oplus \Z$-filtered chain complex obtained by reversing the roles of the Alexander and algebraic filtrations is filtered chain homotopic to the original complex, by \cite[Section 3.5]{OSknots}.

There are certain families of knots for which the knot Floer complex $\CFKi$ is completely determined by classical data. One is quasi-alternating knots, a generalization of alternating knots due to Ozsv\'ath-Szab\'o \cite{OSbranched}. For a quasi-alternating knot $K$, we have that $\tau(K) = -\frac{\sigma(K)}{2}$ (see Section \ref{sec:tau} for the definition of $\tau$) and that $\widehat{HFK}(K)$ is supported in a single diagonal with respect to the Alexander and Maslov gradings \cite{MOqa}. This uniquely specifies $\widehat{HFK}(K)$, and Petkova \cite[Theorem 4]{Petkova} further shows that this uniquely determines $\CFKi(K)$.

An \emph{L-space knot} is a knot $K \subset S^3$ which admits a positive L-space surgery. By \cite[Theorem 1.2]{OSlens}, the knot Floer homology of an L-space knot $K$ is completely determined by the Alexander polynomial of $K$. Moreover, grading considerations imply that in fact the full knot Floer complex $\CFKi(K)$ is determined by $\Delta_K(t)$; see, for example, \cite[Section 2.3]{HHN13} or \cite[Section 7]{HendricksManolescu} for a statement of this result.

\section{Concordance invariants}\label{sec:concinvts}

\subsection{The Ozsv\'ath-Szab\'o $\tau$ invariant}\label{sec:tau}
Given a $\Z$-filtered chain complex with total homology of $\F$, one can ask what is the minimum filtration level in which the homology is supported. In the case where the $\Z$-filtered chain complex in question is $\widehat{CFK}(K)$, this yields the Ozsv\'ath-Szab\'o $\tau$ invariant \cite{OS4ball}. Namely,
\[ \tau(K) = \min \{s \mid \iota \co C \{ i=0, j \leq s \} \rightarrow C\{ i=0 \} \textup{ induces a surjection on homology} \}. \]
Recall that $C\{ i=0 \} = \widehat{CF}(S^3)$ and $H_*(C\{ i=0\}) \cong \F$.

The $\tau$ invariant can also be defined in terms of $HFK^-(K)$: 
\begin{equation}\label{eqn:tauminus}
	\tau(K) = - \max \{ s \in \Z \mid \exists \; \theta \in HFK^-(K, s) \textup{ such that for all } n \in \N, \; U^n \theta \neq 0\}.
\end{equation}
See \cite[Appendix A]{OST} for the proof that these two definitions agree. Ozsv\'ath-Szab\'o \cite{OS4ball} show that 
\begin{itemize}
	\item $\tau(K_1 \# K_2) = \tau(K_1) + \tau(K_2)$, 
	\item $\tau(-K)=-\tau(K)$,
	\item $\tau(K)=0$ if $K$ is slice;
\end{itemize}
that is, $\tau$ induces a homomorphism from $\cC \rightarrow \Z$.

\begin{theorem}[\cite{OS4ball}]\label{thm:tau}
The invariant $\tau$ satisfies the following properties:
\begin{enumerate}
	\item The map $\tau \co \cC \rightarrow \Z$ is a surjective homomorphism.
	\item \label{it:taubound} The absolute value of $\tau$ gives a lower bound on the smooth four-ball genus, i.e., 
	\[ |\tau(K)| \leq g_4(K). \]
	\item If $K$ is quasi-alternating, then $\tau(K) = - \frac{\sigma(K)}{2}$.
	\item Let $T_{p,q}$ denote the $(p,q)$-torus knot. For $p,q \geq1$, we have $\tau(T_{p,q})=\frac{(p-1)(q-1)}{2}=g(T_{p,q})$.
	\item Let $K_+$ be a knot in $S^3$, and $K_-$ the new knot obtained by changing one positive crossing in $K_+$ to a negative crossing. Then 
		\[ \tau(K_+)-1 \leq \tau(K_-) \leq \tau(K_+). \]
\end{enumerate}
\end{theorem}

\begin{remark}
Recall that Rasmussen defined an analogous invariant, $s(K)$, coming from Khovanov homology \cite{Rs} and showed  that $\frac{s(K)}{2}$ also satisfies the above properties. However, Hedden-Ording \cite{HeddenOrding} showed that $\tau(K) \neq \frac{s(K)}{2}$ in general.
\end{remark}

\begin{example}
We have that $\tau(T_{2,3})=1$. In Figure \ref{fig:CFKhatT23}, we can see that $a$ is a generator for $\widehat{HF}(S^3)$, and there is no other generator for $\widehat{HF}(S^3)$ of lower Alexander grading (denoted by the vertical axis). Equivalently, in Figure \ref{fig:CFKminusT23}, we see  that $U^n [c]\neq 0$ for all $n \in \N$, and that $[c] \in HFK^-(K, -1)$ is the element of maximal Alexander grading with that property.
\end{example}

The proof of Theorem \ref{thm:tau}\eqref{it:taubound} relies on certain properties of maps on $\widehat{HF}$ induced by cobordisms, which we discuss in the next subsection.

\subsection{Invariants related to Heegaard Floer homology of surgery along $K$}\label{subsec:HFsurgeryinvts}
Several concordance invariants can be extracted from maps on $HF^+$, $HF^-$, and $\widehat{HF}$ induced by the two-handle cobordism corresponding to integral surgery along $K$.

\subsubsection{Cobordism maps on $\widehat{HF}$}\label{sec:hatmaps}
We begin by considering the two-handle cobordism associated to sufficiently large $N$ surgery along $K$. (In general, $N \geq 2g(K)-1$ is large enough \cite{OSgenusbounds}.) Consider the chain map
\[ \widehat{v}_s \co C\{ \max(i, j-s)=0 \} \rightarrow C\{i=0\} \] 
consisting of quotienting by $C\{ i \leq 0, j=s\}$ followed by inclusion.

Let $W_N(K)$ denote the two-handle cobordism from $S^3$ to $S^3_N(K)$. Consider the cobordism $-W_N(K)$ from $S^3_N(K)$ to $S^3$ (that is, turn $W_N(K)$ upside down and reverse its orientation) with spin$^c$ structure $\mft$ such that 
\[ \langle c_1(\mft), [\hat{F}] \rangle +N = 2s \]
where $\hat{F}$ is a capped-off Seifert surface for $K$.

\begin{theorem}[{\cite[Theorem 4.4]{OSknots}, see also \cite[Theorem 2.3]{OSinteger}}]\label{thm:largeNhat}
For $|s| \leq N/2$, the complex $C\{ \max(i, j-s)=0 \}$ represents $\widehat{CF}(S^3_N(K), \mfs)$ and the map
\[ \widehat{v}_{s,*} \co \widehat{HF}(S^3_N(K), \mfs) \rightarrow \widehat{HF}(S^3) \]
is induced by the two-handle cobordism $-W_N(K)$ with spin$^c$ structure $\mft$ where $\mft|_{S^3_N(K)}=\mfs$.
\end{theorem}


Recall that $\widehat{HF}(S^3) \cong \F$. The map $\widehat{v}_{s,*}$ is trivial for $s < \tau(K)$ and non-trivial for for $s > \tau(K)$; this fact can be deduced from the definition of $\tau$ as in \cite[Proposition 3.1]{OS4ball}. Since $\widehat{HF}(S^3) \cong \F$, the map $\widehat{v}_{s,*}$ is non-trivial if and only if it is surjective.

The concordance invariant $\nu(K)$ is defined in \cite[Section 9]{OSrational} to be
\[ \nu(K) = \min \{ s \mid \widehat{v}_{s,*} \textup{ is surjective} \}. \]
It follows that $\nu(K)$ is equal to either $\tau(K)$ or $\tau(K)+1$.

Closely related to $\nu(K)$ is the $\{-1, 0, 1\}$-valued concordance invariant $\varep$ \cite{Homsmooth}.

\begin{definition}
The invariant $\varep$ is defined as follows:
\begin{itemize}
	\item If $\nu(K)=\tau(K)+1$, the $\varep(K) = -1$.
	\item If $\nu(K)=\tau(K)$ and $\nu(-K)=\tau(-K)$, then $\varep(K)=0$.
	\item If $\nu(-K)=\tau(-K)+1$, then $\varep(K)=1$.
\end{itemize}
\end{definition}
\noindent It follows from \cite[Section 3]{Homsmooth} that these three cases are exhaustive and mutually exclusive.

\begin{figure}[htb!]
\centering
\vspace{5pt}
\subfigure[]{
\begin{tikzpicture}[scale=1.0]
	\filldraw[black!20!white] (-0.15, -2) rectangle (0.15, -0.85);
	\filldraw[black!20!white] (-2, -1.15) rectangle (0.15, -0.85);
	\shade[top color=white, top color=black!20!white] (-0.15, -2) rectangle (0.15, -3);
	\shade[top color=white, right color=black!20!white] (-3, -1.15) rectangle (-2, -0.85);
	
	\begin{scope}[thin, gray]
		\draw [<->] (-3, 0) -- (2, 0);
		\draw [<->] (0, -3) -- (0, 2);
	\end{scope}
	\draw[step=1, black!30!white, very thin] (-2.9, -2.9) grid (1.9, 1.9);

	\filldraw (0, 1) circle (2pt) node[] (a2){};
	\filldraw (0, 0) circle (2pt) node[] (b2){};
	\filldraw (1, 0) circle (2pt) node[] (c2){};
	\draw [very thick, ->] (a2) -- (b2);
	\draw [very thick, ->] (c2) -- (b2);
	
	\filldraw (-1, 0) circle (2pt) node[] (a3){};
	\filldraw (-1, -1) circle (2pt) node[] (b3){};
	\filldraw (0, -1) circle (2pt) node[] (c3){};
	\draw [very thick, ->] (a3) -- (b3);
	\draw [very thick, ->] (c3) -- (b3);
	\node [right] at (c3) {$x_j$};
	
	\filldraw (-2, -1) circle (2pt) node[] (a4){};
	\filldraw (-2, -2) circle (2pt) node[] (b4){};
	\filldraw (-1, -2) circle (2pt) node[] (c4){};
	\draw [very thick, ->] (a4) -- (b4);
	\draw [very thick, ->] (c4) -- (b4);
	
	\filldraw (-2.5, -2.5) circle (0.8pt);
	\filldraw (-2.6, -2.6) circle (0.8pt);
	\filldraw (-2.7, -2.7) circle (0.8pt);

	\filldraw (1.2, 1.2) circle (0.8pt);
	\filldraw (1.3, 1.3) circle (0.8pt);
	\filldraw (1.4, 1.4) circle (0.8pt);
\end{tikzpicture}
}
\hspace{10pt}
\subfigure[]{
\begin{tikzpicture}[scale=1.0]
	\filldraw[black!20!white] (-0.15, -2) rectangle (0.15, 0.15);
	\filldraw[black!20!white] (-2, -0.15) rectangle (0.15, 0.15);
	\shade[top color=white, top color=black!20!white] (-0.15, -2) rectangle (0.15, -3);
	\shade[top color=white, right color=black!20!white] (-3, -0.15) rectangle (-2, 0.15);
	
	\begin{scope}[thin, gray]
		\draw [<->] (-3, 0) -- (2, 0);
		\draw [<->] (0, -3) -- (0, 2);
	\end{scope}
	\draw[step=1, black!30!white, very thin] (-2.9, -2.9) grid (1.9, 1.9);

	\filldraw (0, 1) circle (2pt) node[] (a2){};
	\filldraw (0, 0) circle (2pt) node[] (b2){};
	\filldraw (1, 0) circle (2pt) node[] (c2){};
	\draw [very thick, ->] (a2) -- (b2);
	\draw [very thick, ->] (c2) -- (b2);
	
	\filldraw (-1, 0) circle (2pt) node[] (a3){};
	\filldraw (-1, -1) circle (2pt) node[] (b3){};
	\filldraw (0, -1) circle (2pt) node[] (c3){};
	\draw [very thick, ->] (a3) -- (b3);
	\draw [very thick, ->] (c3) -- (b3);
	\node [right] at (c3) {$x_j$};
	
	\filldraw (-2, -1) circle (2pt) node[] (a4){};
	\filldraw (-2, -2) circle (2pt) node[] (b4){};
	\filldraw (-1, -2) circle (2pt) node[] (c4){};
	\draw [very thick, ->] (a4) -- (b4);
	\draw [very thick, ->] (c4) -- (b4);
	
	\filldraw (-2.5, -2.5) circle (0.8pt);
	\filldraw (-2.6, -2.6) circle (0.8pt);
	\filldraw (-2.7, -2.7) circle (0.8pt);

	\filldraw (1.2, 1.2) circle (0.8pt);
	\filldraw (1.3, 1.3) circle (0.8pt);
	\filldraw (1.4, 1.4) circle (0.8pt);
\end{tikzpicture}
\label{fig:}
}
\caption{Left, $\widehat{v}_{-1,*}$ is trivial. Right, $\widehat{v}_{0,*}$ is non-trivial. It follows that $\nu(K)=0$. Since $\tau(K)=-1$, we have that $\varep(K)=-1$.}
\label{fig:nuexample}
\end{figure}
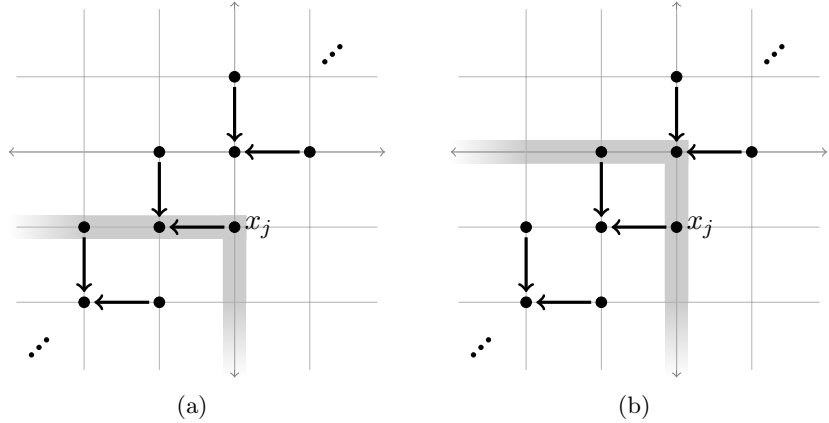

The invariant $\varep$ has an interpretation in terms of a certain type of simplified basis for $\CFKi$. We say a basis $\{x_i\}$ for $\CFKi$ is \emph{vertically simplified} if for each basis element $x_i$ exactly one of the following holds:
\begin{itemize}
	\item $\d^\vert x_i = x_j$ for some $x_j$,
	\item $x_i$ is in the kernel but not the image of $\d^\vert$,
	\item $x_i = \d^\vert x_j$ for a unique $x_j$.
\end{itemize}
In other words, the vertical differential cancels basis elements in pairs, except for a distinguished basis element which is a generator for the vertical homology. We define a \emph{horizontally simplified} basis analogously. Then by \cite[Lemmas 3.2 and 3.3]{HomTau}, there exists a horizontally simplified basis $\{x_i\}$ for $\CFKi(K)$ such that some $x_j$ is the distinguished generator of a vertically simplified basis. Then 
\begin{itemize}
	\item $\varep(K)=-1$ if $x_j$ is not in the kernel of the horizontal differential,
	\item $\varep(K)=0$ if $x_j$ is in the kernel but not the image of the horizontal differential,
	\item $\varep(K)=1$ if $x_j$ is in the image of the horizontal differential.
\end{itemize}
See Figure \ref{fig:epsilonexamples} for examples. 

\begin{remark}\label{rmk:epbasis}
Note that the first bullet point above implies that if $\varep(K)=-1$, then there is a horizontal arrow from $x_j$, and the third bullet point above implies that if $\varep(K)=1$, then there is a horizontal arrow going to $x_j$. The second bullet point above implies that if $\varepsilon(K)=0$, then $x_j$ has no incoming or outgoing horizontal or vertical arrows.
\end{remark}

\begin{figure}[htb!]
\centering
\vspace{5pt}
\subfigure[]{
\begin{tikzpicture}[scale=0.89]
	\begin{scope}[thin, gray]
		\draw [<->] (-2, 0) -- (2, 0);
		\draw [<->] (0, -2) -- (0, 2);
	\end{scope}
	\draw[step=1, black!30!white, very thin] (-1.9, -1.9) grid (1.9, 1.9);
	
	\filldraw (0, 1) circle (2pt) node[] (a2){};
	\filldraw (0, 0) circle (2pt) node[] (b2){};
	\filldraw (1, 0) circle (2pt) node[] (c2){};
	\draw [very thick, ->] (a2) -- (b2);
	\draw [very thick, ->] (c2) -- (b2);
	
	\filldraw (-1, 0) circle (2pt) node[] (a3){};
	\filldraw (-1, -1) circle (2pt) node[] (b3){};
	\filldraw (0, -1) circle (2pt) node[] (c3){};
	\draw [very thick, ->] (a3) -- (b3);
	\draw [very thick, ->] (c3) -- (b3);
	\node [right] at (c3) {$x_j$};
	
	\filldraw (-1.5, -1.5) circle (0.8pt);
	\filldraw (-1.6, -1.6) circle (0.8pt);
	\filldraw (-1.7, -1.7) circle (0.8pt);

	\filldraw (1.2, 1.2) circle (0.8pt);
	\filldraw (1.3, 1.3) circle (0.8pt);
	\filldraw (1.4, 1.4) circle (0.8pt);
\end{tikzpicture}
}
\hspace{10pt}
\subfigure[]{
\begin{tikzpicture}[scale=0.89]
	\begin{scope}[thin, gray]
		\draw [<->] (-2, 0) -- (2, 0);
		\draw [<->] (0, -2) -- (0, 2);
	\end{scope}
	\draw[step=1, black!30!white, very thin] (-1.9, -1.9) grid (1.9, 1.9);
	
	\filldraw (1, 1) circle (2pt) node[] (b1){};
	\filldraw (0, 0) circle (2pt) node[] (b2){};
	\node [right] at (b2) {$x_j$};	
	\filldraw (-1, -1) circle (2pt) node[] (b3){};
	
	\filldraw (-1.5, -1.5) circle (0.8pt);
	\filldraw (-1.6, -1.6) circle (0.8pt);
	\filldraw (-1.7, -1.7) circle (0.8pt);

	\filldraw (1.5, 1.5) circle (0.8pt);
	\filldraw (1.6, 1.6) circle (0.8pt);
	\filldraw (1.7, 1.7) circle (0.8pt);
\end{tikzpicture}
\label{fig:}
}
\hspace{10pt}
\subfigure[]{
\begin{tikzpicture}[scale=0.89]
	\begin{scope}[thin, gray]
		\draw [<->] (-2, 0) -- (2, 0);
		\draw [<->] (0, -2) -- (0, 2);
	\end{scope}
	\draw[step=1, black!30!white, very thin] (-1.9, -1.9) grid (1.9, 1.9);
	
	\filldraw (0, 1) circle (2pt) node[] (a2){};
	\filldraw (1, 1) circle (2pt) node[] (b2){};
	\filldraw (1, 0) circle (2pt) node[] (c2){};
	\draw [very thick, <-] (a2) -- (b2);
	\draw [very thick, <-] (c2) -- (b2);
	\node [left] at (a2) {$x_j$};
	
	\filldraw (-1, 0) circle (2pt) node[] (a3){};
	\filldraw (0, 0) circle (2pt) node[] (b3){};
	\filldraw (0, -1) circle (2pt) node[] (c3){};
	\draw [very thick, <-] (a3) -- (b3);
	\draw [very thick, <-] (c3) -- (b3);
	
	\filldraw (1.5, 1.5) circle (0.8pt);
	\filldraw (1.6, 1.6) circle (0.8pt);
	\filldraw (1.7, 1.7) circle (0.8pt);

	\filldraw (-1.2, -1.2) circle (0.8pt);
	\filldraw (-1.3, -1.3) circle (0.8pt);
	\filldraw (-1.4, -1.4) circle (0.8pt);
\end{tikzpicture}
\label{fig:}
}
\caption{Left, $\varepsilon=-1$. Center, $\varepsilon=0$. Right, $\varepsilon=1$.}
\label{fig:epsilonexamples}
\end{figure}
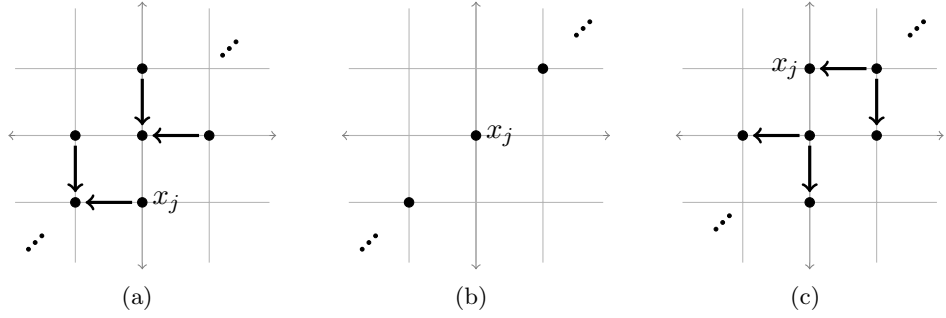

\begin{proposition}[{\cite[Proposition 3.6]{HomTau}}]
The invariant $\varep$ satisfies the following properties:
\begin{enumerate}
	\item If $K$ is slice, then $\varep(K)=0$.
	\item \label{it:eptau} If $\varep(K)=0$, then $\tau(K)=0$.
	\item $\varep(-K) = -\varep(K)$.
	\item If $K$ is quasi-alternating (or more generally, homologically thin), then $\varep(K)=\sgn \tau(K)$.
	\item If $g(K) = |\tau(K)|$, then $\varep(K)=\sgn \tau(K)$.
	\item If $\varep(K_1)=\varep(K_2)$, then $\varep(K_1 \# K_2) = \varep(K_1)$. If $\varep(K_1)=0$, then  $\varep(K_1 \# K_2) = \varep(K_2)$.
\end{enumerate}
\end{proposition}

\begin{remark}
The converse to \eqref{it:eptau} is false; for example, if $K$ is the $(2, -3)$-cable (where the longitudinal winding is $2$) of the right-handed trefoil, then $\tau(K)=0$ but $\varep(K)=1$.
\end{remark}

The invariant $\varep$ is a key ingredient in the following theorem about $\cC_{TS}$, the subgroup of topologically slice knots.

\begin{theorem}[\cite{HomInfiniteRank}]\label{thm:Homsummand}
The group $\cC_{TS}$ contains a direct summand isomorphic to $\Z^\infty$.
\end{theorem}

\noindent Theorem \ref{thm:Homsummand} was reproved by Ozsv\'ath-Stipsicz-Szab\'o \cite{OSS} using the invariant $\Upsilon_K(t)$ described in Section \ref{sec:Upsilon}. The proof in \cite{HomInfiniteRank} relies on considering the group
\begin{equation}\label{eqn:cCFK}
	\cCFK := \Big(\{ \CFKi(K) \mid K \subset S^3\}, \#\Big)/ \sim_\varep
\end{equation}
where $C_1 \sim_\varep C_2$ if $\varep(C_1 \otimes C_2^*)=0$; we call this relation \emph{$\varep$-equivalence}. The group $\cCFK$ is totally ordered:
\[ [C_1] > [C_2] \qquad \textup{ if and only if } \qquad \varep(C_1 \otimes C_2^*) = 1. \]
Understanding the order type of $\cCFK$ and the homomorphism $\cC \rightarrow \cCFK$ defined by $[K] \mapsto [\CFKi(K)]$ are two of the main ideas behind the proof of Theorem \ref{thm:Homsummand}

\begin{remark}
Rather than $\{ \CFKi(K) \mid K \subset S^3\}$, we could instead consider the set of $\Z \oplus \Z$-filtered chain complexes of $\F[U]$-modules satisfying the symmetry and rank properties of $\CFKi$. It follows from \cite[Theorem 7]{HeddenWatson} that these two sets are different. However, it is an open question whether the quotients of these sets by $\sim_\varep$ are equal.
\end{remark}

The equivalence relation $\sim_\varep$ provides one approach to extracting as much concordance information as possible from $\CFKi$. As we will see below, there is still more concordance information in $\CFKi$ than what is captured by the $\varep$-equivalence class $[\CFKi(K)]$.

\subsubsection{Cobordism maps on $HF^-$}
We can similarly consider two-handle cobordism maps on $HF^+$ or $HF^-$, rather than on $\widehat{HF}$. By the argument in \cite[Proposition 2.3]{OSS}, the concordance information coming from $HF^+$ is identical to that coming from $HF^-$; see also \cite[Lemma 2.6]{HomLidman}. Thus, it is sufficient to consider the maps on, say, $HF^-$, as we do here.

As above, we consider sufficiently large $N$ surgery along $K$, where $N \geq 2g(K)-1$ suffices. Consider the chain map
\[ v^-_s \co C\{ \max(i, j-s) \leq 0 \} \rightarrow C\{i \leq 0\} \] 
defined by inclusion. The analogous result to Theorem \ref{thm:largeNhat} holds \cite[Theorem 4.4]{OSknots}. Namely, the complex $C\{ \max(i, j-s) \leq 0 \}$ represents $CF^-(S^3_N(K), \mfs)$ and the map 
\[ v^-_{s,*} \co HF^-(S^3_N(K), \mfs) \rightarrow HF^-(S^3) \]
is induced by the two-handle cobordism $-W_N(K)$ with spin$^c$ structure $\mft$ where $\mft |_{S^3_N(K)}=\mfs$. Note that for $s$ sufficiently large (by \cite{OSgenusbounds}, $s \geq g(K)$ suffices), the map $v^-_{s,*}$ is an isomorphism.

The concordance invariant $\nu^-(K)$ is
\[ \nu^-(K) = \min \{ s \mid v^-_{s,*} \textup{ is surjective} \}. \]
By \cite[Proposition 2.3]{OSS}, the invariant $\nu^-(K)$ is equal to the invariant $\nu^+(K)$ defined in \cite{HomWu}. Rasmussen \cite[Corollary 7.4]{R} shows that
\[ \nu^-(K) \leq g_4(K) \]
and the author and Wu \cite[Proposition 2.3]{HomWu} show that $\nu^-(K) \geq 0$ and
\begin{equation}\label{eqn:taunu}
	\tau(K) \leq \nu(K) \leq \nu^-(K).
\end{equation}
Moreover, the gap between $\tau$ and $\nu^-$ can be made arbitrarily large \cite[Theorem 1]{HomWu}: for any positive integer $n$, there exists a knot $K_n$ with $\tau(K_n) \geq 0$ such that
\[ \tau(K_n)+n \leq \nu^-(K_n) = g_4(K_n). \]
The invariant $\nu^-$ also satisfies a crossing change inequality \cite[Theorem 1.3]{BCG}:
\[ \nu^-(K_+) -1 \leq \nu^-(K_-) \leq \nu^-(K_+) \]
and is subadditive \cite[Theorem 1.4]{BCG}:
\[ \nu^-(K_1 \# K_2) \leq \nu^-(K_1) + \nu^-(K_2). \]

The maps $v^-_{s,*}$ yield further concordance invariants beyond $\nu^-$. Define
\[ V_s = \rank_\F (\coker v^-_{s,*}). \]
The sequence $\{V_s \}$ is non-increasing \cite[Lemma 2.4]{NiWu} and, together with the $d$-invariants of lens spaces, determines the $d$-invariants of Dehn surgery along $K$ \cite[Proposition 1.6]{NiWu}. In particular, the sequence $\{V_s\}$ is a concordance invariant and $V_0=-\frac{1}{2}d(S^3_1(K))$.  The concordance invariants $d(S^3_{\pm 1}(K))$ are studied by Peters in \cite{Peters}. Note that $\nu^- = \min \{ s \mid V_s = 0\}$ and so 
\begin{equation}\label{eqn:nuV0}
	\nu^-(K)=0  \qquad \textup{ if and only if } \qquad V_0(K)=0.
\end{equation}

\begin{proposition}\label{prop:basis}
If $V_0(K) = V_0(-K) = 0$, then there exists a basis for $\CFKi(K)$ with a basis element $x_j$ which generates the homology $HFK^\infty(K)$ and splits off as a direct summand of $\CFKi(K)$. In other words,
\[ \CFKi(K) \simeq \CFKi(U) \oplus A, \]
where $U$ denotes the unknot and A is some acyclic complex.
\end{proposition}

We first give a reformulation of $V_0(-K)$. Consider the chain map
\[ v'^+_s \co C\{ i \geq 0 \} \rightarrow C\{ \min(i, j-s) \geq 0 \} \]
consisting of quotienting by $C\{ j < s\}$; by \cite[Theorem 4.1]{OSknots} this induces the two-handle cobordism map for sufficiently large \emph{negative} surgery in a certain spin$^c$ structure.
Define
\[ V'_s = \rank_\F(\ker v'^+_{s,*}). \]

\begin{remark}
At times, we will be interested in the maps $v^-_{s,*}$ for both $K$ and $-K$, and so we write $v^-_{s,*}(K)$ or $v^-_{s,*}(-K)$ to indicate which knot we are considering. We will also write $v^-_{s,*}([x_i])$ to denote the image of $[x_i]$ under $v^-_{s,*}$. We hope that this abuse of notation serves to clarify, rather than confuse, the reader.
\end{remark}

\begin{lemma}\label{lem:Vprime}
We have that $V_s(K) = V'_{-s}(-K)$.
\end{lemma}

\begin{proof}
By \cite[Section 3.5]{OSknots}, we have that $\CFKi(-K)$ is filtered chain homotopy equivalent to $\CFKi(K)^*$, the dual of $\CFKi(K)$; see \cite[Section  2.1]{OSS} for the definition of the filtration on the dual complex. Given a depiction of $\CFKi(K)$ in the plane, as in Figure \ref{fig:epsilonexamples}, recall that we can obtain $\CFKi(K)^*$ from $\CFKi(K)$ by rotating the page $180^\circ$ and reversing the directions of all of the arrows. It follows that $v'^+_{-s,*}(-K)$ is dual to $v^-_{s,*}(K)$, hence $\rank(\ker v'^+_{-s,*}(-K)) = \rank (\coker v^-_{s,*}(K))$, as desired.
\end{proof}

We will use the following definition in the proof of Proposition \ref{prop:basis}. A \emph{filtered basis} for $\CFKi(K)$ is a basis $\{x_i\}$ for $\CFKi(K)$ with the property that $\{x_i\}$ induces a basis for the associated bigraded complex 
\[ \bigoplus_{(a,b) \in \Z \oplus \Z} C(a,b)\] 
where
\[ C(a,b)=C\{i \leq a, j \leq b\} / \big\langle C\{i \leq a, j \leq b-1 \} \cup C\{i \leq a-1, j \leq b \} \big\rangle. \]
We will also be interested in filtered change of basis; that is, a change of basis from one filtered basis to another. An example of a filtered change of basis is to replace a basis element $x_i$ with $x_i + x_j$, where $x_j$ has bifiltration level less than or equal to that of $x_i$. See \cite[Section 11.5]{LOT} for further discussion of filtered bases.

\begin{proof}[Proof of Proposition \ref{prop:basis}]
We will show that there exists a filtered basis for $\CFKi(K)$ with a distinguished basis element $x_j$ which has no incoming or outgoing arrows. Then $x_j$ will generate the desired direct summand.

If $V_0(K) = 0$, then $\nu^-(K)=0$ by \eqref{eqn:nuV0}. In particular, $v^-_{0,*}(K)$ is surjective, that is, there is a cycle $x_j \in C\{ \max (i, j) \leq 0 \}$ such that $v^-_{0,*}([x_j]) $ is a generator for $H_*(C\{ i \leq 0 \}) \cong F[U]$. Consider the chain map
\[ C\{ i \leq 0\} \rightarrow C\{ i \geq 0\} \]
induced by quotienting by $C\{ i < 0\}$ followed by inclusion. Under the map induced on homology, $1 \in H_*(C\{ i \leq 0 \}) \cong \F[U]$ is mapped to $1 \in H_*(C\{ i \geq 0 \}) \cong \F[U^{-1}]$. Therefore, since the cycle $x_j$ generates $H_*(C\{ i \leq 0 \})$, the image of $x_j$ in $H_*(C\{ i \geq 0 \})$ is non-trivial. By hypothesis, $V_0(-K)=0$, and so Lemma \ref{lem:Vprime} implies that $V'_0(K)=0$, i.e., $v'^+_{0,*}(K)$ in injective. In particular, $v'^+_{0,*}([x_j])$ is non-zero, and so $[x_j]$ is non-zero in $C(0,0)$.

It follows that we can find a filtered basis $\{x_i\}$ containing $x_j$ as a basis element. Since $[x_j]$ is non-zero in $H_*(C\{ \min(i,j) \geq 0\})$, it follows that $x_j$ is not in the image of the induced differential on $C\{ \min(i,j) \geq 0\}$. In particular, if there is an incoming arrow to $x_j$ from, say, a basis element $x_k$, then there is also another arrow from $x_k$ to some basis element $x_\ell$ in $C\{ \min(i,j) \geq 0\}$ such that bifiltration level of $x_\ell$ is greater than or equal to that of $x_j$. Thus, we can perform a filtered change of basis by replacing $x_\ell$ with $x_\ell + x_j$. This change of basis has the effect of removing the arrow from $x_k$ to $x_j$. Repeating this process as necessary, we may remove all of the incoming arrows to $x_j$. Since $x_j$ is a cycle, there are no outgoing arrows from $x_j$, and we have removed all of the incoming arrows. This completes the proof of the proposition.
\end{proof}

\begin{proof}[Proof of Theorem \ref{thm:main}]
Suppose that $K_1$ and $K_2$ are concordant. Then $K_1 \# -K_2$ is slice, so Proposition \ref{prop:basis} implies that
\[ \CFKi(K_1 \# -K_2) \simeq \CFKi(U) \oplus A, \]
for some acyclic complex $A$. Then 
\begin{align*}
	 \CFKi(K_1 \# -K_2 \# K_2) &\simeq (\CFKi(U) \oplus A) \otimes \CFKi(K_2) \\
	 &\simeq \CFKi(K_2) \oplus A' 
\end{align*}
for acyclic $A' = A \otimes \CFKi(K_2)$. Similarly, since $-K_2 \# K_2$ is slice, we have
\begin{align*}
	 \CFKi(K_1 \# -K_2 \# K_2) &\simeq \CFKi(K_1)\otimes \CFKi(-K_2 \# K_2) \\
	 &\simeq \CFKi(K_1)\otimes (\CFKi(U) \oplus A'') \\
	 &\simeq \CFKi(K_1) \oplus A'''.
\end{align*}
This concludes the proof of the theorem.
\end{proof}

\begin{remark}
Note that the invariants $\varep$ and $V_0$ are both invariants of suitably defined bifiltered chain complexes, and that $V_0(C)=V_0(C^*)=0$ implies that $\varep(C)=0$. The converse is false; see Figure \ref{fig:OSScomplex} which depicts the example with $\varep(C) = 0$ but $V_0(C) \neq 0$ from \cite[Proposition 9.4]{OSS}. Thus, rather than $\varep$-equivalence, one may consider the stronger equivalence relation on bifiltered chain complexes:
\[ C_1 \sim C_2 \qquad \textup{ if and only if } \qquad V_0(C_1 \otimes C_2^*)=V_0(C_1^* \otimes C_2)=0. \]
However, unlike the group $\cCFK$ in \eqref{eqn:cCFK}, the resulting group is not naturally totally ordered. Moreover, while it is known that there are complexes (e.g., Figure \ref{fig:OSScomplex}) for which $V_0(C) \neq 0$ but $\varep(C) = 0$, it remains open whether such complexes can be realized as $\CFKi$ of a knot $K \subset S^3$.
\end{remark}

In light of Proposition \ref{prop:basis} (see also Corollary \ref{cor:V0}), one might ask whether there are concordance obstructions which can be non-vanishing on knots with $V_0(K)=V_0(-K)=0$. Recent work of Hendricks and Manolescu \cite{HendricksManolescu} answers this question in the affirmative. They use the conjugation symmetry on the Heegaard Floer complex to define involutive Heegaard Floer homology, corresponding to $\Z_4$-equivariant Seiberg-Witten Floer homology. They obtain two new invariants of homology cobordism, $\dl$ and $\du$, satisfying
\[ \dl(Y, \mfs) \leq d(Y, \mfs) \leq \du(Y, \mfs). \]
They also obtain two new knot concordance invariants, $\Vl$ and $\Vu$, satisfying
\[ \Vu \leq V_0 \leq \Vl .\]
Unlike $\tau$, $\varep$, and $\nu$, which all vanish on knots of finite concordance order, the involutive invariants can be non-zero on such knots; for example, the figure-eight knot has $\Vl=1$ (and $V_0=\Vu=0$) \cite[Section 8.2]{HendricksManolescu}.

\begin{remark}
As far as the author knows, it remains an open question whether there are knots of finite concordance order for which $\nu^-$ and $V_0$ are non-vanishing. However, on the algebraic level, there do exist complexes $C$ with $\nu^-(C)$ and $V_0(C)$ non-zero but $\nu^-(C \otimes C)=V_0(C \otimes C)=0$, for example, the complex $C$ in Figure 3 of \cite{HomInfiniteRank}. (To see that $V_0(C \otimes C)=0$, note that $C$ is isomorphic to $C^*$.)
\end{remark}

Hendricks and Manolescu show that for homologically thin knots and L-space knots, the invariants $\Vl$ and $\Vu$ are completely determined by classical data (for thin knots, the Alexander polynomial and signature, and for L-space knots, the Alexander polynomial) \cite[Sections 7 and 8]{HendricksManolescu}.

\subsection{$d$-invariants of branched covers}
We can also consider the $d$-invariants of cyclic branched covers along $K$ rather than surgery. We will be interested in $\Sigma_{p^n}(K)$, the $p^n$-fold cyclic branched covers along $K$, where $p$ is a prime, since such covers are always rational homology spheres. Note that $\Sigma_{p^n}(K)$ has a canonical spin structure (see \cite[Definition 2.3]{Jabuka}), which we denote by $\mfs_0$.

Manolescu-Owens \cite{ManolescuOwens} consider $\delta(K)=2d(\Sigma_2(K), \mfs_0)$, and Jabuka \cite{Jabuka} considers the more general $\delta_{p^n}(K)=2d(\Sigma_{p^n}(K), \mfs_0)$. Hence the Manolescu-Owens $\delta$ is equal to $\delta_2$. We have that
\[ \delta_{p^n} \co \cC \rightarrow \Z \]
is a surjective concordance homomorphism. Jabuka \cite[Theorem 1.2]{Jabuka} shows that for a fixed $p$,
\[ \bigoplus_{n=1}^\infty \delta_{p^n} \co \cC \rightarrow \Z^\infty \]
is of infinite rank. Livingston \cite{Livingstoncomp} uses $\tau$, $s$, and $\delta$ to show that $\cC_{TS}$, the topologically slice subgroup of $\cC$, splits off a rank three direct summand, c.f., Theorem \ref{thm:Homsummand}.

In general, the invariant $\delta_{p^n}$ is difficult to compute. When $\Sigma_{p^n}(K)$ can described as surgery on a knot $J$ in $S^3$, one can use the knot Floer complex of $J$ together with \cite[Proposition 1.6]{NiWu} to compute $\delta_{p^n}(K)$; see, for example, \cite{HeddenLivRub} and \cite{Jabuka}. Moreover, similar techniques can be employed when the branched cover is surgery on a link; see \cite{HeddenKimLiv}. Alternatively, if $\Sigma_{p^n}(K)$ can be described as the boundary of a negative definite plumbing, then the results of \cite{OSplumbed} and \cite{Nemethi} (see also \cite{CanKarakurt}) can be used to compute $\delta_{p^n}(K)$. See \cite{ManolescuOwens} and \cite{Jabuka} for examples of applications of these techniques.


\subsection{Ozsv\'ath-Stipsicz-Szab\'o's $\Upsilon(t)$ invariant}\label{sec:Upsilon}
The invariant $\varep$ and the techniques of \cite{HomInfiniteRank}, as well as the invariants $\nu$, $\nu^-$, and $V_s$, are all attempts to extract tractable concordance invariants from $\CFKi$. Continuing in this vein is the Ozsv\'ath-Stipsicz-Szab\'o $\Upsilon_K(t)$ invariant, which takes values in the group of piecewise-linear functions on $[0,2]$. See also Livingston's notes on $\Upsilon_K(t)$ \cite{LivingstonUpsilon}.

\begin{theorem}[\cite{OSS}]
The invariant $\Upsilon_K(t)$ satisfies the following properties:
\begin{enumerate}
	\item $\Upsilon_K(2-t) = \Upsilon_K(t)$.
	\item $\Upsilon_K(0)=0$.
	\item $\Upsilon'_K(0) = -\tau(K)$.
	\item \label{it:Upsilonhom} $\Upsilon_{K_1 \# K_2}(t) = \Upsilon_{K_1}(t) + \Upsilon_{K_2}(t)$.
	\item $\Upsilon_{-K}(t) = - \Upsilon_{K}(t)$.
	\item \label{it:Upsilon4genus}For $t \in [0,1]$, we have $ | \Upsilon_K(t) | \leq t g_4(K)$.
	\item Let $K_+$ be a knot in $S^3$, and $K_-$ the new knot obtained by changing one positive crossing in $K_+$ to a negative crossing. Then for $t \in [0,1]$
		\[ \Upsilon_{K_+}(t) \leq \Upsilon_{K_-}(t) \leq \Upsilon_{K_+}(t) + 1. \]
\end{enumerate}
\end{theorem}

Note that it follows from \eqref{it:Upsilonhom}--\eqref{it:Upsilon4genus} above that $\Upsilon$ is a homomorphism from the concordance group to the group of piecewise-linear functions on $[0,2]$. Other properties of $\Upsilon$ are that the derivative of $\Upsilon$ has only finitely many discontinuities, and that the derivative is always an integer \cite[Proposition 1.4]{OSS}.

\begin{remark}
Property \eqref{it:Upsilon4genus} above follows from the inequality
\[ |\Upsilon_K(t)| \leq t \max (\nu^-(K), \nu^-(-K) ). \]
In particular, the four-ball genus bound given by $\Upsilon_K(t)$ is no better than the bound coming from $\nu^-(\pm K)$.
\end{remark}

We briefly sketch the definition of $\Upsilon$. Fix a rational number $t=\frac{p}{q} \in [0,2]$. (See \cite[Section 3]{OSS} for the extension to the case of irrational $t$.) We will work over the ring $\F[v^{1/q}]$ where $v^2=U$. Consider the complex $tCFK(K)$ generated by the usual Heegaard Floer generators. The differential is defined to be
\[ \d_t \bfx = \sum_{\bfy \in \bbT_\alpha \cap \bbT_\beta} \sum_{\{ \phi \in \pi_2(\bfx, \bfy) \mid \mu(\phi)=1\} } \# \Big(\frac{\cM(\phi)}{\R}\Big) v^{tn_z(\phi)+(2-t)n_w(\phi)} \bfy. \]
The grading is defined to be
\[ \gr_t(\bfx) = M(\bfx) - tA(\bfx). \]
Multiplication by $v$ lowers $\gr_t$ by one, i.e., for $n \in \Z$
\[ \gr_t(v^{n/q} \bfx) = \gr_t(\bfx) - \frac{n}{q}. \]
The homology of $tCFK(K)$ is called \emph{$t$-modified knot Floer homology} and is denoted by $tHFK(K)$. Recall from \eqref{eqn:tauminus} that $\tau$ can be defined in terms of the maximal grading of a non-$U$-torsion element in $HFK^-$. Similarly, Ozsv\'ath-Stipsicz-Szab\'o \cite[Definition 3.6]{OSS} define $\Upsilon_K(t)$ to be the maximal $\gr_t$-grading of any homogenous non-$v$-torsion element of $tHFK(K)$.

\begin{remark}
Note that when $t=0$, we have $\d_t=\d^-$, and when $t=2$, we have effectively reversed the roles of $z$ and $w$.
\end{remark}

\begin{example}
We consider the diagram in Figure \ref{fig:HDtrefoil} for the right-handed trefoil $T_{2,3}$. The differential is
\[ 	\d_t a = 0 \qquad \quad	\d_t b = v^{2-t} a + v^t c \qquad \quad	\d_t c = 0, \]
and the gradings are
\[	\gr_t(a)=-t \qquad \quad	\gr_t(b)=-1 \qquad \quad	\gr_t(c)=-2+t. \] 
We see that both $[a]$ and $[c]$ are non-$v$-torsion elements of $tHFK(K)$, with $\gr_t(a) \geq \gr_t(c)$ for $0 \leq t \leq 1$ and $\gr_t(c) \geq \gr_t(a)$ for $1 \leq t \leq 2$. Thus,
\[
\Upsilon_{T_{2,3}}(t) = 
		\begin{cases}
			-t \quad & \text{for } 0 \leq t \leq 1 \\
			-2+t \quad & \text{for } 1 < t \leq 2.
		\end{cases}
\]
\end{example}

\section{Concordance genus bounds}\label{sec:concgen}
Recall that the \emph{concordance genus $g_c(K)$} of a knot $K$ is the minimal Seifert genus of any knot $K'$ which is concordant to $K$. Recall also that the knot Floer complex detects genus \cite{OSgenusbounds}; a knot $K$ with knot Floer complex $C=\CFKi(K)$ has genus
\[ g(C) = \max \{ j \mid H_*(C(0, j)) \neq 0\}. \]
Since concordant knots are $\varep$-equivalent, the invariant $\gamma(K)$ defined in \cite{HomConcGen} to be
\[ \gamma(K) = \min \{ g(C') \mid C' \sim_\varep \CFKi(K) \} \]
provides a lower bound on the concordance genus of $K$. The invariant $\gamma$ can be used to show that there are topologically slice knots with $g_4(K)=1$ but $g_c(K)$ arbitrarily large \cite[Theorem 3]{HomConcGen}; the analogous result for general knots was proven by Nakanishi \cite{Nakanishi} and for algebraically slice knots by Livingston \cite{Livingstonconcordancegenus}.

The invariant $\Upsilon_K(t)$ also provides a concordance genus bound. Let $s$ denote the maximum of the finitely many values of $\Upsilon'_K(t)$. Then by \cite[Theorem 1.13]{OSS}, 
\[ s \leq g_c(K). \]


\section{Comparison of the invariants}\label{sec:comparison}
We give a comparison of some of the invariants discussed above. We begin with $\varep$ and $\Upsilon(t)$. Proposition 9.4 of \cite{OSS} gives an example of a bifiltered chain complex $C$ with $\Upsilon_C(t) \not\equiv 0 $ but $\varep(C)=0$; see Figure \ref{fig:OSScomplex}. In the opposite direction, there exist knots $K$ for which $\varep(K) \neq 0$ but $\Upsilon_K(t) \equiv 0$ \cite{HomEpUp}. 

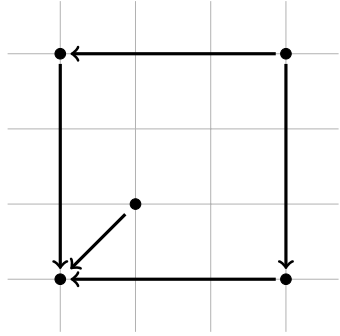
\begin{figure}[ht]
\centering
\begin{tikzpicture}[scale=1.0]
	\draw[step=1, black!30!white, very thin] (-0.7, -0.7) grid (3.7, 3.7);
	
	\filldraw (3, 3) circle (2pt) node[] (a){};
	\filldraw (3, 0) circle (2pt) node[] (b){};
	\filldraw (0, 3) circle (2pt) node[] (c){};
	\filldraw (0, 0) circle (2pt) node[] (e){};
	\filldraw (1, 1) circle (2pt) node[] (x){};
	\draw [very thick, <-] (b) -- (a);
	\draw [very thick, <-] (c) -- (a);
	\draw [very thick, <-] (e) -- (b);
	\draw [very thick, <-] (e) -- (c);
	\draw [very thick, <-] (e) -- (x);
\end{tikzpicture}
\caption{The complex $C$ from Proposition 9.4 of \cite{OSS}. It is unknown whether there is a knot $K$ for which $\CFKi(K) \simeq C$.}
\label{fig:OSScomplex}
\end{figure}

The invariants $\tau, \nu, \nu^-, V_s, \gamma, \varep$, and $\Upsilon(t)$ are all determined by $\CFKi$. 
Furthermore, we have the following corollary of Proposition \ref{prop:basis}:

\begin{corollary}\label{cor:V0}
If $V_0(K)=V_0(-K)=0$, then $\tau, \nu, \nu^-, V_s, \gamma, \varep$, and $\Upsilon(t)$ all vanish, i.e., agree with their values on the unknot.
\end{corollary}

\begin{proof}
Proposition \ref{prop:basis} implies that if $V_0(K)=V_0(-K)=0$, then $\CFKi(K) \simeq \CFKi(U) \oplus A$ for some acyclic complex $A$. The result now follows from the definitions of $\tau, \nu, \nu^-, V_s, \gamma, \varep$, and $\Upsilon(t)$.
\end{proof}

The Hendricks-Manolescu concordance invariants $\Vl$ and $\Vu$ are notable in that they need not vanish when $V_0(K)=V_0(-K)=0$ \cite[Section 8.2]{HendricksManolescu}. A natural question to ask is whether $\Vl$ and $\Vu$ are completely determined by $\CFKi$. As far as the author knows, this question remains open, although the answer is yes for certain families of knots: Hendricks and Manolescu \cite[Sections 7 and 8]{HendricksManolescu} show that $\Vl$ and $\Vu$ are indeed completely determined by $\CFKi$ for homologically thin knots (which includes all alternating knots) and L-space knots.

In addition to Corollary \ref{cor:V0}, recall that $\varep(K)=0$ implies $\tau(K)=0$, and $\Upsilon_K(t)=0$ implies that $\tau(K)=0$. As can be seen in Table \ref{tab:comparison}, the reverse implications do not hold.
We write $T_{m,n;p,q}$ to denote the $(p,q)$-cable of the $(m,n)$-torus knot, where $p$ denotes the longitudinal winding.

\begin{table}[htb!]
\begin{center}
\begin{tabular}{*{7}{@{\hspace{15pt}}c}}
\hline
 & $\tau$ & $\nu$ & $\nu^-$ & $V_0$ & $\varep$ & $\Upsilon(t)$ \\ \hline
 &&&&&& \\
$\CFKi(-T_{2,5}\# T_{4,5} \# -T_{2,3;2,5})$ & $0$ & $1$ & $1$ & $1$ & $-1$ & $\equiv 0$ \\ 
&&&&&& \\
$C$ & 0 & 0 & 2 & 2 & 0 & {$\left\{\! \openup-0.75\jot \begin{aligned} 
               2-3t &\textup{ if } 2/3 \leq t < 1 \\    
               -4+3t &\textup{ if } 1 \leq t < 4/3 \\  
               0 &\textup{ otherwise} \end{aligned}\right.$} \\ \hline
\end{tabular}
\end{center}
\caption{A comparison of some concordance invariants associated to $\CFKi$. The complex $C$ is depicted in Figure \ref{fig:OSScomplex}.}
\label{tab:comparison}
\end{table}


\bibliographystyle{amsalpha}

\bibliography{mybib}

\providecommand{\bysame}{\leavevmode\hbox to3em{\hrulefill}\thinspace}
\providecommand{\MR}{\relax\ifhmode\unskip\space\fi MR }
\providecommand{\MRhref}[2]{%
  \href{http://www.ams.org/mathscinet-getitem?mr=#1}{#2}
}
\providecommand{\href}[2]{#2}
\begin{thebibliography}{Hom15b}

\bibitem[BCG15]{BCG}
J\'ozsef Bodn\'ar, Daniele Celoria, and Marco Golla, \emph{A note on cobordisms
  of algebraic knots}, 2015, preprint, arXiv:1509.08821.

\bibitem[CK14]{CanKarakurt}
M.~B. Can and {\c{C}}.~Karakurt, \emph{Calculating {H}eegaard-{F}loer homology
  by counting lattice points in tetrahedra}, Acta Math. Hungar. \textbf{144}
  (2014), no.~1, 43--75.

\bibitem[Ghi08]{Ghiggini}
Paolo Ghiggini, \emph{Knot {F}loer homology detects genus-one fibred knots},
  Amer. J. Math. \textbf{130} (2008), no.~5, 1151--1169.

\bibitem[HHN13]{HHN13}
Stephen Hancock, Jennifer Hom, and Michael Newman, \emph{On the knot {F}loer
  filtration of the concordance group}, J. Knot Theory Ramifications
  \textbf{22} (2013), no.~14, 1350084, 30.

\bibitem[HKL12]{HeddenKimLiv}
Matthew Hedden, Se-Goo Kim, and Charles Livingston, \emph{Topologically slice
  knots of smooth concordance order two}, 2012, preprint, arXiv:1212.6628.

\bibitem[HL15]{HomLidman}
Jennifer Hom and Tye Lidman, \emph{A note on positive-definite, symplectic
  four-manifolds}, 2015, preprint, arXiv:1510.00373. To appear on J. Eur. Math.
  Soc. (JEMS).

\bibitem[HLR12]{HeddenLivRub}
Matthew Hedden, Charles Livingston, and Daniel Ruberman, \emph{Topologically
  slice knots with nontrivial {A}lexander polynomial}, Adv. Math. \textbf{231}
  (2012), no.~2, 913--939.

\bibitem[HM15]{HendricksManolescu}
Kristen Hendricks and Ciprian Manolescu, \emph{Involutive {F}loer homology},
  2015, preprint, arXiv:1507.00383.

\bibitem[HO08]{HeddenOrding}
Matthew Hedden and Philip Ording, \emph{The {O}zsv\'ath-{S}zab\'o and
  {R}asmussen concordance invariants are not equal}, Amer. J. Math.
  \textbf{130} (2008), no.~2, 441--453.

\bibitem[Hom14a]{HomTau}
Jennifer Hom, \emph{Bordered {H}eegaard {F}loer homology and the tau-invariant
  of cable knots}, J. Topol. \textbf{7} (2014), no.~2, 287--326.

\bibitem[Hom14b]{Homsmooth}
\bysame, \emph{The knot {F}loer complex and the smooth concordance group},
  Comment. Math. Helv. \textbf{89} (2014), no.~3, 537--570.

\bibitem[Hom15a]{HomInfiniteRank}
\bysame, \emph{An infinite-rank summand of topologically slice knots}, Geom.
  Topol. \textbf{19} (2015), no.~2, 1063--1110.

\bibitem[Hom15b]{HomConcGen}
\bysame, \emph{On the concordance genus of topologically slice knots}, Int.
  Math. Res. Not. IMRN (2015), no.~5, 1295--1314.

\bibitem[Hom16]{HomEpUp}
\bysame, \emph{A note on the concordance invariants epsilon and upsilon}, Proc.
  Amer. Math. Soc. \textbf{144} (2016), no.~2, 897--902.

\bibitem[HW14]{HeddenWatson}
Matthew Hedden and Liam Watson, \emph{On the geography and botany of knot
  {F}loer homology}, 2014, preprint, arXiv:1404.6913.

\bibitem[HW16]{HomWu}
Jennifer Hom and Zhongtao Wu, \emph{Four-ball genus bounds and a refinement of
  the {O}zs\'vath-{S}zab\'o tau-invariant}, J. Symplectic Geom. \textbf{14}
  (2016), no.~1, 305--323.

\bibitem[Jab12]{Jabuka}
Stanislav Jabuka, \emph{Concordance invariants from higher order covers},
  Topology Appl. \textbf{159} (2012), no.~10-11, 2694--2710.

\bibitem[Liv04a]{Livingstoncomp}
Charles Livingston, \emph{Computations of the {O}zsv\'ath-{S}zab\'o knot
  concordance invariant}, Geom. Topol. \textbf{8} (2004), 735--742
  (electronic).

\bibitem[Liv04b]{Livingstonconcordancegenus}
\bysame, \emph{The concordance genus of knots}, Algebr. Geom. Topol. \textbf{4}
  (2004), 1--22.

\bibitem[Liv05]{Livingston}
\bysame, \emph{A survey of classical knot concordance}, Handbook of knot
  theory, Elsevier B. V., Amsterdam, 2005, pp.~319--347.

\bibitem[Liv14]{LivingstonUpsilon}
\bysame, \emph{Notes on the knot concordance invariant {U}psilon}, 2014,
  preprint, arXiv:1412.0254.

\bibitem[LOT08]{LOT}
Robert Lipshitz, Peter Ozsv\'ath, and Dylan Thurston, \emph{Bordered {H}eegaard
  {F}loer homology: {I}nvariance and pairing}, 2008, preprint,
  arXiv:0810.0687v4.

\bibitem[Man14]{Manolescuknots}
Ciprian Manolescu, \emph{An introduction to knot floer homology}, 2014,
  preprint, arXiv:1401.7107.

\bibitem[MO07]{ManolescuOwens}
Ciprian Manolescu and Brendan Owens, \emph{A concordance invariant from the
  {F}loer homology of double branched covers}, Int. Math. Res. Not. IMRN
  (2007), no.~20.

\bibitem[MO08]{MOqa}
Ciprian Manolescu and Peter Ozsv{\'a}th, \emph{On the {K}hovanov and knot
  {F}loer homologies of quasi-alternating links}, Proceedings of {G}\"okova
  {G}eometry-{T}opology {C}onference 2007, G\"okova Geometry/Topology
  Conference (GGT), G\"okova, 2008, pp.~60--81.

\bibitem[Nak81]{Nakanishi}
Yasutaka Nakanishi, \emph{A note on unknotting number}, Math. Sem. Notes Kobe
  Univ. \textbf{9} (1981), no.~1, 99--108.

\bibitem[N{\'e}m05]{Nemethi}
Andr{\'a}s N{\'e}methi, \emph{On the {O}zsv\'ath-{S}zab\'o invariant of
  negative definite plumbed 3-manifolds}, Geom. Topol. \textbf{9} (2005),
  991--1042.

\bibitem[Ni07]{Ni}
Yi~Ni, \emph{Knot {F}loer homology detects fibred knots}, Invent. Math.
  \textbf{170} (2007), no.~3, 577--608.

\bibitem[Ni09]{Nifibred3}
\bysame, \emph{Heegaard {F}loer homology and fibred 3-manifolds}, Amer. J.
  Math. \textbf{131} (2009), no.~4, 1047--1063.

\bibitem[NW10]{NiWu}
Yi~Ni and Zhongtao Wu, \emph{Cosmetic surgeries on knots in {$S^3$}}, 2010,
  preprint, arXiv:1009.4720.

\bibitem[OS03a]{OSabsgr}
Peter Ozsv{\'a}th and Zolt{\'a}n Szab{\'o}, \emph{Absolutely graded {F}loer
  homologies and intersection forms for four-manifolds with boundary}, Adv.
  Math. \textbf{173} (2003), no.~2, 179--261.

\bibitem[OS03b]{OS4ball}
\bysame, \emph{Knot {F}loer homology and the four-ball genus}, Geom. Topol.
  \textbf{7} (2003), 615--639.

\bibitem[OS03c]{OSplumbed}
\bysame, \emph{On the {F}loer homology of plumbed three-manifolds}, Geom.
  Topol. \textbf{7} (2003), 185--224 (electronic).

\bibitem[OS04a]{OSsurvey}
\bysame, \emph{Heegaard diagrams and holomorphic disks}, Different faces of
  geometry, Int. Math. Ser. (N. Y.), vol.~3, Kluwer/Plenum, New York, 2004,
  pp.~301--348.

\bibitem[OS04b]{OSgenusbounds}
\bysame, \emph{Holomorphic disks and genus bounds}, Geom. Topol. \textbf{8}
  (2004), 311--334.

\bibitem[OS04c]{OSknots}
\bysame, \emph{Holomorphic disks and knot invariants}, Adv. Math. \textbf{186}
  (2004), no.~1, 58--116.

\bibitem[OS04d]{OS3manifolds2}
\bysame, \emph{Holomorphic disks and three-manifold invariants: properties and
  applications}, Ann. of Math. (2) \textbf{159} (2004), no.~3, 1159--1245.

\bibitem[OS04e]{OS3manifolds1}
\bysame, \emph{Holomorphic disks and topological invariants for closed
  three-manifolds}, Ann. of Math. (2) \textbf{159} (2004), no.~3, 1027--1158.

\bibitem[OS05a]{OSunknotting}
\bysame, \emph{Knots with unknotting number one and {H}eegaard {F}loer
  homology}, Topology \textbf{44} (2005), no.~4, 705--745.

\bibitem[OS05b]{OSlens}
\bysame, \emph{On knot {F}loer homology and lens space surgeries}, Topology
  \textbf{44} (2005), no.~6, 1281--1300.

\bibitem[OS05c]{OSbranched}
\bysame, \emph{On the {H}eegaard {F}loer homology of branched double-covers},
  Adv. Math. \textbf{194} (2005), no.~1, 1--33.

\bibitem[OS06a]{OSsmoothfour}
\bysame, \emph{Holomorphic triangles and invariants for smooth four-manifolds},
  Adv. Math. \textbf{202} (2006), no.~2, 326--400.

\bibitem[OS06b]{OSsurvey2}
\bysame, \emph{Lectures on {H}eegaard {F}loer homology}, Floer homology, gauge
  theory, and low-dimensional topology, Clay Math. Proc., vol.~5, Amer. Math.
  Soc., Providence, RI, 2006, pp.~29--70.

\bibitem[OS08]{OSinteger}
Peter~S. Ozsv{\'a}th and Zolt{\'a}n Szab{\'o}, \emph{Knot {F}loer homology and
  integer surgeries}, Algebr. Geom. Topol. \textbf{8} (2008), no.~1, 101--153.

\bibitem[OS11]{OSrational}
\bysame, \emph{Knot {F}loer homology and rational surgeries}, Algebr. Geom.
  Topol. \textbf{11} (2011), no.~1, 1--68.

\bibitem[OSS14]{OSS}
Peter Ozsv\'ath, Andr\'as Stipsicz, and Zolt{\'a}n Szab{\'o}, \emph{Concordance
  homomorphisms from knot {F}loer homology}, 2014, preprint, arXiv:1407.1795.

\bibitem[OST08]{OST}
Peter Ozsv{\'a}th, Zolt{\'a}n Szab{\'o}, and Dylan Thurston, \emph{Legendrian
  knots, transverse knots and combinatorial {F}loer homology}, Geom. Topol.
  \textbf{12} (2008), no.~2, 941--980.

\bibitem[Pet10]{Peters}
Thomas~David Peters, \emph{Computations of {H}eegaard {F}loer homology: {T}orus
  bundles, {L}-spaces, and correction terms}, Ph.D. thesis, Columbia
  University, 2010.

\bibitem[Pet13]{Petkova}
Ina Petkova, \emph{Cables of thin knots and bordered {H}eegaard {F}loer
  homology}, Quantum Topol. \textbf{4} (2013), no.~4, 377--409.

\bibitem[Ras03]{R}
Jacob Rasmussen, \emph{Floer homology and knot complements}, Ph.D. thesis,
  Harvard University, 2003.

\bibitem[Ras10]{Rs}
\bysame, \emph{Khovanov homology and the slice genus}, Invent. Math.
  \textbf{182} (2010), no.~2, 419--447.

\end{thebibliography}

\end{document}